\documentclass{amsart}
\usepackage{amssymb}
\usepackage[ruled,vlined]{algorithm2e}
\usepackage{url}
\usepackage{booktabs}
\usepackage{multirow}

\newcommand{\assign}{\leftarrow}
\newcommand{\ep}{\varepsilon}
\newcommand{\ZZ}{\mathbf{Z}}
\newcommand{\divides}{\mathbin|}
\newcommand{\ndivides}{\mathbin{\nmid}}

\newtheorem{thm}{Theorem}
\newtheorem{prop}[thm]{Proposition}

\newtheorem*{lem*}{Lemma}
\theoremstyle{remark}
\newtheorem*{remark}{Remark}

\begin{document}

\title{A multimodular algorithm for computing Bernoulli numbers}
\date{13th October 2008}
\author{David Harvey}
\address{Courant Institute of Mathematical Sciences, New York University, 251 Mercer St, New York NY 10012, USA}
\email{dmharvey@cims.nyu.edu}
\urladdr{http://cims.nyu.edu/\~{}harvey/}

\begin{abstract}
We describe an algorithm for computing Bernoulli numbers. Using a parallel implementation, we have computed $B_k$ for $k = 10^8$, a new record. Our method is to compute $B_k$ modulo $p$ for many small primes $p$, and then reconstruct $B_k$ via the Chinese Remainder Theorem. The asymptotic time complexity is $O(k^2 \log^{2+\ep} k)$, matching that of existing algorithms that exploit the relationship between $B_k$ and the Riemann zeta function. Our implementation is significantly faster than several existing implementations of the zeta-function method.
\end{abstract}

\maketitle

\section{Introduction}
\label{sec:introduction}

The Bernoulli numbers $B_0, B_1, B_2, \ldots$ are rational numbers defined by
\[
 \frac{x}{e^x - 1} = \sum_{k=0}^\infty \frac{B_k}{k!} x^k.
\]
In this paper we focus on the problem of computing $B_k$, as an exact rational number, for a single large (even) value of $k$. To date, the most efficient algorithms for computing an isolated $B_k$ have been based on the formula
\begin{equation}\label{eq:bern-zeta}
 B_k = (-1)^{k/2+1} \frac{2 \zeta(k) k!}{(2\pi)^k}, \qquad \text{$k \geq 2$ and $k$ even}, 
\end{equation}
where $\zeta(s)$ is the Riemann zeta function \cite[p.~231]{ireland}. One first computes a high-precision approximation for $B_k$ via \eqref{eq:bern-zeta}, using the Euler product for $\zeta(s)$. The exact denominator of $B_k$ is known by the von Staudt--Clausen theorem. Provided that the numerical approximation to $B_k$ has enough correct bits --- it turns out that $O(k \log k)$ bits is enough --- we may recover the numerator of $B_k$ exactly. We will refer to this algorithm and its variants as the `zeta-function algorithm'.

The zeta-function algorithm has been rediscovered numerous times. The formula \eqref{eq:bern-zeta} itself goes back to Euler. Chowla and Hartung \cite{chowla-hartung} give an `exact formula' for $B_k$, the evaluation of which is tantamount to executing the algorithm, but using the defining sum for $\zeta(k)$ instead of the Euler product, and using $2(2^k - 1)$ in place of the exact denominator of $B_k$. Fillebrown \cite{fillebrown} mentions \cite{chowla-hartung} and points out that using the Euler product is asymptotically faster; she gives a detailed time and space complexity analysis, and attributes the algorithm to Herbert Wilf (unpublished). Fee and Plouffe \cite{fee-plouffe} state that they discovered and implemented the zeta-function algorithm in 1996. Kellner \cite{kellner} also refers to \cite{chowla-hartung} and suggests using the Euler product. He mentions some data produced by Fee and Plouffe, although not their algorithm. According to Pavlyk \cite{pavlyk}, the implementation in Mathematica 6 \cite{mathematica} derives from Kellner's work. Stein \cite[p.~30]{stein} indicates another independent discovery by Henri Cohen, Karim Belabas, and Bill Daly (unpublished), that led to the implementation in PARI/GP \cite{pari}.

In this paper we present a new algorithm for computing $B_k$ that does not depend on \eqref{eq:bern-zeta} or any properties of $\zeta(s)$. The basic idea is to compute $B_k$ modulo $p$ for sufficiently many small primes $p$, and then reconstruct $B_k$ using the Chinese Remainder Theorem. Using a parallel implementation, we have computed $B_k$ for $k = 10^8$, improving substantially on the previous record of $k = 10^7$ \cite{pavlyk}.

Let $M(n)$ denote the time (bit operations) required to multiply two $n$-bit integers. We may assume that $M(n) = O(n \log^{1+\ep} n)$, using for example the Sch\"onhage--Strassen algorithm \cite{schonhage-strassen}. According to Fillebrown's analysis \cite[p.~441]{fillebrown}, the zeta-function algorithm for computing $B_k$ requires $O(k)$ multiplications of integers with $O(k \log k)$ bits, so its running time is $O(k^2 \log^{2+\ep} k)$. The new algorithm also runs in time $O(k^2 \log^{2+\ep} k)$. However, we will see that over the feasible range of computation the running time behaves more like $O(k^2 \log k)$.

The new algorithm is easily parallelisable, as the computations for the various $p$ are independent. Only the final stages of the modular reconstruction are more difficult to execute in parallel, but these account for only $O(k^{1+\ep})$ of the running time. During the modular computations, each thread requires only $O(\log^{1+\ep} k)$ space. The zeta-function algorithm may also be parallelised, for instance by mapping Euler factors to threads, but this requires $O(k \log k)$ space per thread.

\section{Computing $B_k$ modulo $p$}
\label{sec:modular}

Throughout this section $p \geq 5$ denotes a prime, and $k$ an even integer satisfying $2 \leq k \leq p-3$. Our algorithm for computing $B_k$ modulo $p$ is based on the following well-known congruence. Let $0 < c < p$, and suppose that $c^k \not\equiv 1 \pmod p$. Then
\begin{equation}\label{eq:sum-first}
 B_k \equiv \frac{k}{1 - c^k} \sum_{x=1}^{p-1} x^{k-1} h_c(x) \pmod p,
\end{equation}
where
 \[ h_c(x) := \frac{(x \bmod p) - c(x/c \bmod p)}{p} + \frac{c-1}2. \]
Equation \eqref{eq:sum-first} may be deduced by reading the statement of Theorem 2.3 of \cite[\S2]{lang-cyclotomic} modulo $p$ (as is done in the proof of Corollary 2 to that theorem).

Equation \eqref{eq:sum-first} may be used to compute $B_k$ modulo $p$ directly, but computing $x^{k-1}$ modulo $p$ for each $x$ is unnecessarily expensive. Instead, we select a generator $g$ of the multiplicative group $(\ZZ/p\ZZ)^\times$, put $c = g$, and rewrite the sum as
\begin{equation}\label{eq:sum-gen}
 B_k \equiv \frac{k}{1 - g^k} \sum_{i=1}^{p-1} g^{i(k-1)} h_g(g^i) \pmod p.
\end{equation}
Note that $g^k \not\equiv 1 \pmod p$ since $p-1 \ndivides k$. For $x \not\equiv 0 \pmod p$ we have $h_c(-x) = -h_c(x)$, and $g^{(p-1)/2} \equiv -1 \pmod p$, so we obtain the half-length sum
\begin{equation}\label{eq:sum-gen-half}
 B_k \equiv \frac{2k}{1 - g^k} \sum_{i=1}^{(p-1)/2} g^{i(k-1)} h_g(g^i) \pmod p,
\end{equation}
leading to the following algorithm.

\begin{algorithm}
\label{algo:simple}
\SetLine
\dontprintsemicolon
$g \assign \text{a generator of $(\ZZ/p\ZZ)^\times$}$\;
$r \assign g^{k-1} \pmod p$\;
$u \assign (g-1)/2 \pmod p$\;
$S \assign 0$, \, $X \assign 1$, \, $Y \assign r$\;
\For{$i \assign 1$ \KwTo $(p-1)/2$}{
$q \assign \lfloor g X / p \rfloor$\;
$S \assign (S + (u - q)Y) \bmod p$\;
$X \assign g X \bmod p$, \, $Y \assign r Y \bmod p$\;
}
\KwRet $2kS / (1 - g^k) \bmod p$\;
\caption{Compute $B_k$ modulo $p$}
\end{algorithm}

\begin{prop}
\label{prop:modular}
Let $p \geq 5$ be a prime, and let $k$ be an even integer in the interval $2 \leq k \leq p-3$. Algorithm \ref{algo:simple} computes $B_k$ modulo $p$ in time $O(p \log^{1+\ep} p)$.
\end{prop}
\begin{proof}
At the top of the loop we have $X = g^{i-1} \pmod p$ and $Y = g^{i(k-1)} \pmod p$. The value assigned to $q$ is $(g(g^{i-1} \bmod p) - (g^i \bmod p))/p = u - h_g(g^i)$. Therefore the value added to $S$ is $g^{i(k-1)} h_g(g^i)$. By \eqref{eq:sum-gen-half} the return value is $B_k$ modulo $p$.

We now analyse the complexity. A generator of $(\ZZ/p\ZZ)^\times$ may be found deterministically in time $O(p^{1/4+\ep})$ \cite{shparlinski}. In the main loop, the algorithm performs $O(p)$ additions, multiplications and divisions (with remainder) of integers with $O(\log p)$ bits. The computation of $r$ requires another $O(\log k) = O(\log p)$ such operations. The division by $1 - g^k$ requires an extended GCD of integers with $O(\log p)$ bits. The total cost is $O(p \log^{1+\ep} p)$.
\end{proof}

\begin{remark}
To compute $B_{10^8}$ (see \S\ref{sec:examples}), the largest prime needed is $1558322063$. Even on a 32-bit machine, this fits into a single machine word, and the cost of arithmetic modulo $p$ does not depend on $p$. Therefore, over the feasible range of computation, the complexity of Algorithm \ref{algo:simple} may be regarded as only $O(p)$.
\end{remark}

\section{Computing $B_k$ as a rational number}
\label{sec:rational}

In this section we present a multimodular algorithm for computing $B_k = N_k/D_k$ as an exact rational number. Before getting to the algorithm proper, we make some preliminary calculations. In what follows, we assume that $k \geq 4$ and that $k$ is even.

The denominator $D_k$ is given by the von Staudt--Clausen theorem \cite[p.~233]{ireland}:
 \[ D_k = \prod_{\substack{\text{$p$ prime} \\ p-1 \divides k}} p. \]
Note that $D_k \leq 2^{k + 1}$, since $D_k$ divides $2(2^k - 1)$ \cite[p.~114]{chowla-hartung}. 

The size of the numerator $N_k$ may be bounded quite tightly as follows. From \eqref{eq:bern-zeta} and Stirling's approximation \cite[p.~120]{lang-analysis} we have
\[
\begin{aligned}
 \log |B_k| & = \log 2 + \log \zeta(k) + \log k! - k \log 2\pi \\
   & \leq \left(k + \frac12\right) \log k - (1 + \log 2\pi)k + (\log 2 + \log \zeta(k) + \log \sqrt{2\pi} ) + \frac{1}{12k}.
\end{aligned}
\]
Since $\zeta(k) \leq \zeta(4)$ and $12k \geq 48$, a short calculation shows that $|N_k| < 2^\beta$ where $\beta := \lceil(k + 0.5) \log_2 k - 4.094 k + 2.470 + \log_2 D_k\rceil$.

For any $X > 0$, let $M_X := \prod_{p \leq X, p - 1 \ndivides k} p$. Our strategy will be to select $X$ so that $M_X \geq 2^\beta$, and then compute $N_k$ modulo $M_X$. This ensures that we have sufficient information to reconstruct $N_k$. In the algorithm itself our choice of $X$ will be quite sharp, but we also need a rough \emph{a priori} estimate. We will take $Y := \max(37, \lceil(k + 0.5) \log_2 k\rceil)$. To check that this works, we will use the estimate $\sum_{p \leq x} \log p \geq 0.73x$, a weak form of the prime number theorem, valid for $x \geq 37$ \cite[p.~111]{narkiewicz}. Then we have
\[
\begin{aligned}
 \log_2 M_Y
     & = -\log_2 D_k + \sum_{p \leq Y} \log_2 p \\
     & \geq -(k+1) + \frac{0.73}{\log 2} Y \\
     & \geq Y - k - 1 \\
     & \geq (k + 0.5) \log_2 k - k - 1 \\
     & \geq (k + 0.5) \log_2 k - 3.094k + 6.470  \quad (\text{since $k \geq 4$}) \\
     & \geq (k + 0.5) \log_2 k - 4.094k + 3.470 + \log_2 D_k + 2\\
     & \geq \lceil(k + 0.5) \log_2 k - 4.094k + 2.470 + \log_2 D_k \rceil + 2 \\
     & = \beta + 2,
\end{aligned}
\]
so that $M_Y \geq 2^{\beta + 2}$. Algorithm \ref{algo:rational} presents the resulting algorithm; several important details are given in the proof of the following theorem.

\begin{algorithm}
\label{algo:rational}
\SetLine
\dontprintsemicolon
\tcp{compute $D_k$ and preliminary bounds}\;
$Y \assign \max(37, \lceil(k + 0.5)\log_2 k \rceil)$\;
Compute list of all primes $p \leq Y$\;
$D_k \assign \prod_{p-1 \divides k} p$\;
$\beta \assign \lceil(k + 0.5) \log_2 k - 4.094 k + 2.470 + \log_2 D_k\rceil$\;
\medskip
\tcp{compute tight bound $X$}\;
$p \assign 3$, \, $M' \assign 1.0$\;
\While{$M' < 2^{\beta + 1}$}{
   $p \assign \mathrm{NextPrime}(p)$\;
   \lIf{$p - 1 \ndivides k$}{$M' \assign p M'$}\;
}
$X \assign p$\;
\medskip
\tcp{collect modular information}\;
\lFor{$p \in \{2, 3, 5, \ldots, X\}, p - 1 \ndivides k$}{$r_p \assign B_k \pmod p$}\;
\medskip
\tcp{reconstruction}\;
$M \assign M_X = \prod_{p \leq X, p - 1 \ndivides k} p$\;
$R \assign $ unique integer in $[0, M)$ congruent to $r_p$ modulo $p$ for all primes $p \divides M$\;
$N' \assign D_k R \pmod M$\;
\lIf{$k \equiv 2 \pmod 4$}{$N_k \assign N'$} \lElse{$N_k \assign N' - M$}\;
\KwRet{$N_k/D_k$}\;
\caption{Compute $B_k$ as a rational number ($k \geq 4$, $k$ even)}
\end{algorithm}

\begin{thm}\label{thm:rational}
Algorithm \ref{algo:rational} computes $B_k$ in time $O(k^2 \log^{2+\ep} k)$.
\end{thm}
\begin{proof}
Computing all primes less than $Y$ requires time $O(k^{1+\ep})$, via an elementary sieving algorithm, since $Y = O(k \log k)$. In what follows we assume that all primality tests and enumeration of primes are performed with the aid of this list.

To compute $D_k$, make a list of the factors of $k$ (for example, by testing divisibility of $k$ by each integer in $[2, \sqrt k]$), and for each factor $d \divides k$ check whether $d+1$ is prime. Multiply together those that are prime using a product tree \cite[Algorithm 10.3, p.~293]{mca}. Since $\log D_k = O(k)$, this takes time $O(k^{1+\ep})$.

The goal of the while-loop is to quickly find a bound $X$, much tighter than $Y$, such that $M_X \geq 2^\beta$. The variable $M'$ holds an approximation to the product of the primes encountered so far. It should be represented in floating-point, with at least $n = \lceil\log_2 Y\rceil + 1$ bits in the mantissa, so that each multiplication by $p$ magnifies the relative error by at most $1 + 2^{-n}$. We claim that the loop terminates before exhausting the precomputed list of primes. Indeed, let $s$ be the number of $p$ such that $p \leq Y$ and $p-1 \ndivides k$. After $s$ iterations, $M'$ approximates $M_Y$ with relative error at most $(1 + 2^{-n})^s \leq (1 + (2Y)^{-1})^Y \leq e^{1/2}$; that is, $M' \geq e^{-1/2}M_Y \geq e^{-1/2} 2^{\beta+2} \geq 2^{\beta + 1}$. This is impossible since the loop termination condition would already have been met earlier. This argument also shows that the selected bound $X$ satisfies $M_X \geq e^{-1/2} 2^{\beta+1} \geq 2^\beta$. Since $Y = O(k \log k)$, computing $X$ takes time $O(k^{1+\ep})$.

In the for-loop, for each $p$ we use Algorithm \ref{algo:simple} to compute $r_p := B_k \pmod p$. Note that Algorithm \ref{algo:simple} requires that $2 \leq k \leq p-3$. To satisfy this requirement, we put $m = k \bmod (p-1)$, compute $B_m \pmod p$ using Algorithm \ref{algo:simple}, and then recover $B_k \pmod p$ via Kummer's congruence $B_k/k \equiv B_m/m \pmod p$
\cite[p.~41]{lang-cyclotomic}. The total number of primes processed is no more than $\pi(Y) = O(Y/\log Y) = O(k \log k / \log(k \log k)) = O(k)$. 
According to Proposition \ref{prop:modular}, the cost for each prime is $O(p \log^{1+\ep} p) = O(k \log^{2 + \ep} k)$, since $p \leq Y = O(k \log k)$. Therefore the total cost of the while-loop is $O(k^2 \log^{2+\ep} k)$.

We compute $M$ and $R$ simultaneously using fast Chinese Remaindering \cite[Theorem 10.25, p.~302]{mca}; the cost is $O((\log M_X)^{1+\ep}) = O((k \log k)^{1+\ep}) = O(k^{1+\ep})$.

In the last few lines we recover $N_k$. By construction we have $R \equiv B_k \pmod M$, so $N' \equiv N_k \pmod M$. Note that $|N_k| < 2^\beta \leq M$. If $k \equiv 2 \pmod 4$ then $N_k$ is positive so we simply have $N_k = N'$; otherwise $N_k$ is negative, and we must correct $N'$ by subtracting $M$.
\end{proof}

\begin{remark}
During the modular collection phase, we skipped those $p$ for which $p - 1 \divides k$. An alternative would be to use the fact that $pB_k \equiv -1 \pmod p$ for these $p$ \cite[p.~233]{ireland}, and to apply the multimodular algorithm directly to $N_k$ rather than to $B_k$. This would require the additional minor overhead of computing $D_k \pmod p$ for all of the other primes.
\end{remark}
\begin{remark}
As mentioned below the proof of Proposition \ref{prop:modular}, the running time of Algorithm \ref{algo:simple} behaves like $O(p)$ in practice. The corresponding bound for the running time of Algorithm \ref{algo:rational} is $O(k^2 \log k)$.
\end{remark}

\section{Implementation techniques for computing $B_k$ modulo $p$}
\label{sec:tricks}

In this section we sketch several techniques that yield a large constant factor improvement in the speed of computing $B_k$ modulo $p$, for almost all $k$ and $p$. In experiments we have found that these techniques yield an increase in speed by a factor in excess of $150$ compared to an efficient implementation of Algorithm \ref{algo:simple}.

We first perform some preparatory algebra. Consider again the congruence \eqref{eq:sum-first}. Let $n$ be the multiplicative order of $c$ in $(\ZZ/p\ZZ)^\times$, and put $m = (p-1)/n$. Let $g$ be a generator of $(\ZZ/p\ZZ)^\times$. Then $\{g^i c^{-j}\}_{0 \leq i < m, 0 \leq j < n}$ forms a complete set of representatives for $(\ZZ/p\ZZ)^\times$. Putting $x = g^i c^{-j}$, we obtain
\begin{equation}\label{eq:double-sum}
  B_k \equiv \frac{k}{1 - c^k} \sum_{\substack{0 \leq i < m \\ 0 \leq j < n}} (g^{k-1})^i (c^{k-1})^{-j} h_c(g^i c^{-j}) \pmod p.
\end{equation}
Considerations similar to those of \S\ref{sec:modular} allow us to halve the number of terms, as follows. First suppose that $n$ is even. Since $h_c(-x) = -h_c(x)$ and $c^{n/2} \equiv -1 \pmod p$, the double sum in \eqref{eq:double-sum} becomes
\[
\begin{aligned}
 & \mathrel{\phantom{\equiv}} \sum_{\substack{0 \leq i < m \\ 0 \leq j < n/2}} (g^{k-1})^i (c^{k-1})^{-j} h_c(g^i c^{-j}) + \sum_{\substack{0 \leq i < m \\ 0 \leq j < n/2}} (g^{k-1})^i (c^{k-1})^{-j-n/2} h_c(g^i c^{-j-n/2}) \\
 & \equiv 2\sum_{\substack{0 \leq i < m \\ 0 \leq j < n/2}} (g^{k-1})^i (c^{k-1})^{-j} h_c(g^i c^{-j}) \pmod p.
\end{aligned}
\]
Now suppose that $n$ is odd. Then $m$ is even, and we have $(-g^{m/2})^n \equiv (-1)^n g^{(p-1)/2} \equiv (-1)(-1) \equiv 1 \pmod p$, 
so that $-g^{m/2} \equiv c^\lambda \pmod p$ for some $0 \leq \lambda < n$. The double sum in \eqref{eq:double-sum} becomes
\[
\begin{aligned}
 & \mathrel{\phantom{\equiv}} \sum_{\substack{0 \leq i < m/2 \\ 0 \leq j < n}} (g^{k-1})^i (c^{k-1})^{-j} h_c(g^i c^{-j}) + \sum_{\substack{0 \leq i < m/2 \\ 0 \leq j < n}} (g^{k-1})^{i+m/2} (c^{k-1})^{-j} h_c(g^{i+m/2} c^{-j}) \\
 & \equiv \sum_{\substack{0 \leq i < m/2 \\ 0 \leq j < n}} (g^{k-1})^i (c^{k-1})^{-j} h_c(g^i c^{-j}) - \sum_{\substack{0 \leq i < m/2 \\ 0 \leq j < n}} (g^{k-1})^i (c^{k-1})^{-j+\lambda} h_c(-g^i c^{-j+\lambda}) \\
 & \equiv 2 \sum_{\substack{0 \leq i < m/2 \\ 0 \leq j < n}} (g^{k-1})^i (c^{k-1})^{-j} h_c(g^i c^{-j}) \pmod p.
\end{aligned}
\]
We combine both cases into a single statement by writing
\begin{equation}\label{eq:bern-combined}
  B_k \equiv \frac{2k}{1 - c^k} \sum_{0 \leq i < m'} (g^{k-1})^i \sum_{0 \leq j < n'} (c^{k-1})^{-j} h_c(g^i c^{-j}) \pmod p,
\end{equation}
where
 \[ n' = \begin{cases} n/2 & \text{$n$ even,} \\ n & \text{$n$ odd,} \end{cases} \qquad m' = \frac{(p-1)/2}{n'}. \]

The sum in \eqref{eq:bern-combined} may be evaluated by an algorithm similar to Algorithm \ref{algo:simple}, with an outer loop over $i$ and an inner loop over $j$; we omit the details. The inner loop is executed $n'm' = (p-1)/2$ times, the same number of iterations as in the main loop of Algorithm \ref{algo:simple}. Note that if $c$ is chosen `randomly', then we expect $n'$ to be quite large, so it is imperative to make the inner loop as efficient as possible.

To this end, we specialise to the case $c = 1/2 \pmod p$; this will allow us to exploit the fact that multiplication by $2$ modulo $p$ is very cheap. Of course, this specialisation does not work if $2^k \equiv 1 \pmod p$. In practice, such primes tend to be rare for any given $k$, and we may fall back on Algorithm \ref{algo:simple} to handle them.

Directly from the definition of $h_c(x)$, we have $h_{1/2}(x) = -f(x)/4$ where
 \[ f(x) := \begin{cases} 1 & \phantom{p/}0 < (x \bmod p) < p/2, \\ -1 & p/2 < (x \bmod p) < p. \end{cases} \]
From \eqref{eq:bern-combined} we obtain
\begin{equation}\label{eq:bern-2}
 B_k \equiv \frac{k}{2(2^{-k} - 1)} \sum_{0 \leq i < m'} (g^{k-1})^i \sum_{0 \leq j < n'} (2^{k-1})^j f(g^i 2^j) \pmod p,
\end{equation}
where $n$ is the order of $2$ in $(\ZZ/p\ZZ)^\times$, and $n'$ and $m'$ are defined as before.

In the algorithm that evaluates \eqref{eq:bern-2}, the inner loop is considerably simpler than the main loop of Algorithm \ref{algo:simple}. Given $g^i 2^j$ at the beginning of the loop, we compute the next term $g^i 2^{j+1}$ by doubling it and subtracting $p$ if necessary, and $f(g^i 2^j)$ is $-1$ or $+1$ according to whether the subtraction occurred. Therefore only a single modular multiplication is required on each iteration, to keep track of $(2^{k-1})^j$.

Next we describe further optimisations for computing a sum of the form
\begin{equation}\label{eq:inner-sum}
 \sum_{0 \leq j < N} r^j f(2^j s) \pmod p.
\end{equation}
(For evaluating \eqref{eq:bern-2}, we take $N = n'$, $r = 2^{k-1}$, $s = g^i$.) The following observation is crucial: if $0 < s < p$, and if the binary expansion of $s/p$ is $0 . b_0 b_1 b_2 \ldots$ (where each $b_j$ is 0 or 1), then $f(2^j s) = (-1)^{b_j}$. Indeed, the binary expansion of $2^j s/p$ is obtained by shifting that of $s/p$ to the left by $j$ digits, so
 \[ (2^j s \bmod p)/p = 0.b_j b_{j+1} \ldots, \]
and then
 \[ f(2^j s) = +1 \iff 0 < (2^j s \bmod p) < p/2 \iff b_j = 0. \]
Assume that we have available a routine for computing the binary expansion of $s/p$, whose output is a sequence of base-$2^m$ digits. Strip off the last $N \bmod m$ terms of \eqref{eq:inner-sum} and compute them separately; we may then assume that $N$ is divisible by $m$. Put $j = mt + u$, so that \eqref{eq:inner-sum} becomes
\begin{equation}\label{eq:split-digits}
 \sum_{0 \leq t < N/m} (r^m)^t \sum_{0 \leq u < m} r^u f(2^{mt+u} s).
\end{equation}
We may now use a caching strategy as follows. Maintain a table $\{T_\sigma\}$ of length $2^m$, where $\sigma \in \{+1, -1\}^m$. Each $T_\sigma$ is a residue modulo $p$, and is initialised to zero. For each $0 \leq t < N/m$, add $(r^m)^t$ to the table element $T_\sigma$, where $\sigma$ is determined from the $t$-th base-$2^m$ digit of $s/p$ by
 \[ \sigma = (f(2^{mt} s), f(2^{mt + 1} s), \ldots, f(2^{mt + m-1} s)). \]
(In practice, $\sigma$ is represented by the sequence of bits themselves, and is obtained by a simple bit-mask.) After finishing the loop over $t$, compute the $2^m$ values $V_\sigma = \sum_{0 \leq u < m} r^u \sigma_u$, where $\sigma = (\sigma_0, \ldots, \sigma_{m-1})$. Then the desired sum is given by $\sum_\sigma V_\sigma T_\sigma$. The main benefit of this approach is that in the inner loop we only perform $N/m$ modular multiplications, instead of $N$. The tradeoff is that we must maintain the table $\{T_\sigma\}$, and some extra work is required at the end to assemble the information from the table. For this to be worthwhile, we should have $N \gg 2^m$. We should also choose $m$ so that the base-$2^m$ digits of $s/p$ can be extracted efficiently. Moreover, memory locality problems can offset the savings in arithmetic if $m$ is too large. We have found in practice that $m = 8$ works quite well.

In the implementation discussed in \S\ref{sec:examples}, we made several further optimisations, which we describe only briefly.

For better memory locality --- and indeed less total memory consumption --- we split \eqref{eq:split-digits} into blocks and process each block separately. This avoids needing to store and then later retrieve too many bits of the expansion of $s/p$. Instead of computing $s/p$ for each $s$ separately, we precompute an approximation to $1/p$, and then multiply it by each $s$ that occurs, taking advantage of the speed of multiplication by a single-word integer compared to the corresponding division.

We use several tables instead of just one. For example, on a 32-bit machine, we use four tables: the highest 8 bits of each word of $s/p$ contribute to the first table, the next 8 bits to the second table, and so on. This reduces further the number of modular multiplications. When adding values into a table, we skip the reduction modulo $p$, delaying this until the end of the loop. Of course, this is only possible when the word size is large enough compared to $p$, otherwise overflow may occur. Finally, instead of re-initialising the table on each iteration of the outer loop, we simply continue to accumulate data into the same table. This partly mitigates against the overhead incurred when $m'$ is unusually large compared to $n'$.

Most of the modular arithmetic uses Montgomery's reduction algorithm \cite{montgomery}.

\section{Performance data}
\label{sec:examples}

The author implemented the above algorithms in a C++ package named bernmm (\emph{Bern}oulli \emph{m}ulti-\emph{m}odular). The source code is released under the GNU General Public License (GPL), and is freely available from the author's web page. It depends on the GMP library \cite{gmp} for arbitrary precision arithmetic, and on NTL \cite{ntl} for some of the single-precision modular arithmetic. It is included as a standard component of the Sage computer algebra system (version 3.1.2) \cite{sage}, accessible via the \texttt{bernoulli} function.

The parallel version distributes the modular computations among the requested number of threads, and depends on the standard pthreads API. For the reconstruction phase, the lowest layers of the subproduct tree are performed in parallel, but the number of threads decreases towards the top of the tree, since the extended GCD routine is not threaded.

Table \ref{tab:timings} shows timing data for $10^4 \leq k \leq 10^8$, spaced by intervals of $\sqrt{10}$, for bernmm and several existing implementations of the zeta-function algorithm. The timings were obtained on a 16-core 2.6GHz AMD Opteron (64-bit) machine with 96 GB RAM, running Ubuntu Linux. The last two rows of the table ($k = \lfloor 10^{7.5} \rfloor$ and $k = 10^8$) improve on the prior record $k = 10^7$ set by Pavlyk \cite{pavlyk}. The values $B_{31622776}$ and $B_{10^8}$ may be downloaded from the author's web page.

The compiler used was gcc 4.1.3. We used the version of GMP shipped with Sage 3.1.2, which is the same as GMP 4.2.2, but also includes assembly code written by Pierrick Gaudry that improves performance on the Opteron, and code by Niels M\"oller that implements a quasi-linear time extended GCD algorithm \cite{moller}. We used NTL version 5.4.1, also included in Sage.

The timings for PARI/GP were obtained for version 2.3.3, using the same patched version of GMP. We used the \texttt{bernfrac} function from the GP interpreter. The timings for calcbn, a specialised program for computing Bernoulli numbers, were obtained from calcbn 2.0 \cite{calcbn}, again using the same patched GMP. The timings for Mathematica were obtained from the Linux x86 (64-bit) version of Mathematica 6.0, using the \texttt{BernoulliB} function. Neither PARI/GP nor Mathematica supply a parallel implementation as far as we are aware.

The peak memory usage for computing $B_{10^8}$ was approximately 5 GB; this occurred during the final extended GCD operation. By comparison, PARI/GP used 4.3 GB for $k = 10^7$, and calcbn used 1.7 GB for $k = 10^7$ (with ten threads).

For each implementation, we observe that the ratio of the times in adjacent rows of the table is approximately 10, reflecting the predicted quasi-quadratic asymptotic running time of both the zeta-function algorithm and the multimodular algorithm.

We checked our results by two methods. First, we checked that as a real number, the proposed value for $B_k$ agrees with the estimate $|B_k| \approx 2(2\pi)^{-k} k!$. Since we computed $B_k$ by modular information alone, this is a very strong consistency check. Second, we reduced modulo $p$ the proposed value for $B_k$, and compared this against the output of Algorithm \ref{algo:simple}, for a few primes \emph{distinct} from those primes that we used to compute $B_k$. Again, this is a very strong consistency check --- stronger in fact, since it involves all of the bits of the proposed $B_k$.

Mathematica, PARI/GP and calcbn use GMP internally, so naturally any improvement in GMP's large integer arithmetic will improve the timings reported in the table. In this connection we mention the improvements in multiplication speed reported by \cite{zimmermann}; their paper also gives comparisons with other implementations of large-integer arithmetic.

\newcommand{\tm}{\hspace{1pt} \mathord{\cdot} \hspace{1pt} }

\begin{table}
\begin{tabular}{ccccccccc}
\toprule
     &        & PARI   & MMA & \multicolumn{2}{c}{calcbn} & \multicolumn{2}{c}{bernmm} \\
&        
     & CPU    & CPU    & CPU    & wall   & CPU     & wall    \\
$k$  & threads
     & time   & time   & time   & time   & time    & time    \\
\midrule
$10^4$
& 1  & 0.07             & 0.18             & 0.04             &                  & 0.25             &                  \\
$\lfloor 10^{4.5} \rfloor$
& 1  & 0.68             & 1.80             & 0.49             &                  & 1.02             &                  \\
$10^5$
& 1  & 7.45             & 21.7             & 5.72             &                  & 5.46             &                  \\
$\lceil 10^{5.5} \rceil$
& 1  & 88.3             & 285              & 82.2             &                  & 38.6             &                  \\
& 2  &                  &                  & 82.3             & 41.6             & 38.6             & 20.4             \\
$10^6$
& 1  & 1058             & 3434             & 935              &                  & 340              &                  \\
& 4  &                  &                  & 967              & 246              & 341              & 96.1             \\
$\lceil 10^{6.5} \rceil$
& 1  & $1.24 \tm 10^4$  & $4.25 \tm 10^4$  &                  &                  & 3397             &                  \\
& 4  &                  &                  & $1.33 \tm 10^4$  & 3385             & 3406             & 909              \\
$10^7$
& 1  & $1.42 \tm 10^5$  & $5.10 \tm 10^5$  &                  &                  &                  &                  \\
& 10 &                  &                  & $1.60 \tm 10^5$  & $1.61 \tm 10^4$  & $3.59 \tm 10^4$  & 3938             \\
$\lfloor 10^{7.5} \rfloor$
& 10 &                  &                  &                  &                  & $3.86 \tm 10^5$  & $4.03 \tm 10^4$  \\
$10^8$
& 10 &                  &                  &                  &                  & $4.19 \tm 10^6$  & $4.27 \tm 10^5$  \\
\bottomrule
\end{tabular}
\caption{Time (in seconds) for computing $B_k$ for PARI/GP, Mathematica, calcbn and bernmm.}
\label{tab:timings}
\end{table}

\section*{Acknowledgments}

Many thanks to Joe Buhler, Bernd Kellner, and Andrew Sutherland for their comments on a draft of this paper, and to the Department of Mathematics at Harvard University for providing the hardware on which the computations were performed.

\bibliographystyle{amsalpha}
\bibliography{bernmm}

\end{document}